\documentclass[11pt,a4paper,twoside]{article}
\usepackage[lmargin=20mm,rmargin=20mm,vmargin=35mm]{geometry}
\usepackage{indentfirst}
\usepackage[cp1250]{inputenc}
\usepackage{amsfonts}
\usepackage{amssymb}
\usepackage{amsmath}
\usepackage{amsthm}
\usepackage{color}

\newtheorem{tw}{Theorem}

\newtheorem{lem}[tw]{Lemma}
\newtheorem{cor}[tw]{Corollary}
\theoremstyle{definition}
\newtheorem{deff}[tw]{Definition}
\newtheorem{ex}[tw]{Example}
\newtheorem{rem}[tw]{Remark}
\newtheorem{pro}[tw]{Problem}

\begin{document}

\title{Asymptotic property C of the countable direct sum of uniformly discrete $0$-hyperbolic spaces}
\author{Kamil Orzechowski\\
Faculty of Mathematics and Natural Sciences\\
University of Rzesz\'ow\\
35-959 Rzesz\'ow, Poland\\
E-mail: kamil.orz@gmail.com}
\date{}

\maketitle

\begin{abstract}
We define the direct sum of a countable family of pointed metric spaces in a way resembling the direct sum of groups. Then we prove that if a family consists of $0$-hyperbolic (in the sense of Gromov) and $D$-discrete spaces, then its direct sum has asymptotic property C. The main example is a countable direct sum of free groups of (possibly varying) finite rank. This is a generalization of T. Yamauchi's result concernig the countable direct sum of the integers.

\vspace{5mm}

\noindent \textbf{Keywords}: asymptotic property C, Gromov hyperbolic spaces, direct sum.

\vspace{5mm}

\noindent \textbf{2010 Mathematics Subject Classification}: Primary 54F45; Secondary 20F67, 20F69.
\end{abstract}

\section{Introduction}

We recall some definitions and known facts about the notion of {\em Gromov hyperbolicity}. We refer to \cite[Section 6.2]{Roe} and \cite[Section 3]{Bonk} for detailed exposition.

Let $(X,d)$ be a metric space with a fixed basepoint $x_0\in X$. For any $x,y\in X$ we denote by $(x|y)_{x_0}$ the {\em Gromov product} of $x$ and $y$ with respect to the basepoint $x_0$ defined by the following formula:
\begin{equation}
(x|y)_{x_0}=\frac{1}{2}\left(d(x,x_0)+d(y,x_0)-d(x,y)\right).
\end{equation}
\begin{deff}\label{hd}
We say that a space $X$ is $\delta$-{\em hyperbolic} for $\delta\geq 0$ if and only if
\begin{equation}\label{hyp}
(x|z)_{x_0}\geq \min \{(x|y)_{x_0},(y|z)_{x_0}\}-\delta
\end{equation}
for any $x,y,z\in X$ and for any choice of a basepoint $x_0\in X$.
We call $X$ (Gromov) {\em hyperbolic} if there exists $\delta\geq 0$ such that $X$ is $\delta$-hyperbolic.
\end{deff}

It can be shown that if a space satisfies \eqref{hyp} for a fixed $x_0$, then it is $2\delta$-hyperbolic. Thus, in the case $\delta=0$ the condition defining $0$-hyperbolicity does not depend on the choice of a basepoint and we will omit it in the notation writing simply $(x|y)$ for the Gromov product.

Bonk and Schramm showed in \cite[Theorem 4.1]{Bonk} that any $\delta$-hyperbolic metric space can be isometrically embedded into a complete $\delta$-hyperbolic geodesic metric space.
For geodesic metric spaces there is a characterisation of hyperbolicity due to Rips using so called {\em thin triangles}. As we restrict our attention to the case $\delta=0$, a geodesic metric space is $0$-hyperbolic if and only if it is uniquely geodesic and for any $x,y,z\in X$ the geodesic $[x,y]$ is contained in the union
$[x,z]\cup [z,y]$. Such spaces are called $\mathbb{R}$-{\em trees}.

Summarising, $0$-hyperbolic spaces are precisely subspaces of $\mathbb{R}$-trees. We prefer the term ``$0$-hyperbolic'' because our arguments will not depend on existence of geodesics, solely on the Gromov condition \eqref{hyp} for $\delta=0$.

\begin{deff}\label{ud}
We call a metric space $(X,d)$ $D$-{\em discrete} if $d(x,y)\geq D$ for all $x\neq y$, $x,y\in X$. If there exists $D>0$ such that $X$ is $D$-discrete, we say that $X$ is {\em uniformly discrete}. 
\end{deff}

\begin{ex}
Each free group $\mathbb{F}_n$ on $n$ generators equipped with the usual word length metric becomes a $1$-discrete $0$-hyperbolic space. If we consider the Cayley graph of $\mathbb{F}_n$ with its edges realized as isometric copies of the unit interval $[0,1]$, we get a geodesic $0$-hyperbolic space (losing uniform discreteness). 
\end{ex}

We wil need some useful definitions.

\begin{deff}\label{ub}
We say that a non-empty family $\mathcal{U}$ of metric spaces is {\em uniformly bounded} if the number
$$\mathrm{mesh}(\mathcal{U}):=\sup \{\mathrm{diam}(U) \colon U \in \mathcal{U}\}$$
is finite. Let $R>0$, we say that $\mathcal{U}$ is $R$-{\em disjoint} if for any different $A,B\in \mathcal{U}$ we have
$$\mathrm{dist}(A,B):=\inf \{d(a,b)\colon a\in A, b\in B\} \geq R.$$ 
\end{deff}

Let us recall the definition of the asymptotic property C, which was introduced by Dranishnikov \cite{Dran}. It is one of several possible generalizations of the {\em asymptotic dimension}.

\begin{deff}
We say that a metric space $X$ has the {\em asymptotic property C} if for any strictly increasing sequence $(a_i)_{i\in \mathbb{N}}$ of natural numbers there is $n\in \mathbb{N}$ and a finite sequence $\{\mathcal{U}_i\}_{i=1}^{n}$ of uniformly bounded families such that $\bigcup_{i=1}^{n} \mathcal{U}_i$ covers $X$ and $\mathcal{U}_i$ is $a_i$-disjoint for $i=1,\dots,n$.
\end{deff}

T. Radul showed in \cite{Rad} that the asymptotic property C can be thought of as a transfinite extension of asymptotic dimension.
T. Yamauchi in his paper \cite{Y} proved that the group $\bigoplus_{i=1}^{\infty} \mathbb{Z}$ with a proper invariant metric has asymptotic property C.
Our goal is to extend this result to a more general context replacing each copy of $\mathbb{Z}$ with a $0$-hyperbolic space, all of them being $D$-discrete with common $D>0$. Of course, we have to define what {\em direct sum} means for pointed metric spaces (not necessarily groups) and choose a metric such that succeeding spaces are rescaled with multipliers tending to infinity.

\section{Some auxiliary facts}

We will restate some properties of $0$-hyperbolic spaces. They are essential for a proof that such spaces have asymptotic dimension at most $1$ (see \cite[Proposition 9.8]{Roe}).

From now, let $X$ be a $0$-hyperbolic space with a fixed basepoint $x_0\in X$. For real numbers $\alpha < \beta$ we define the {\em interval} $[\alpha,\beta)\subset X$ by the formula
$$[\alpha,\beta)=\{x\in X \colon \alpha \leq d(x,x_0) < \beta\}.$$
Note that we allow $\alpha,\beta$ to be negative, in particular the interval may be empty. It will not bother us and will be useful for keeping some notation concise. 
Fix $\alpha \in\mathbb{R}$ and $d,r>0$. Let us define an equivalence relation $\sim_{\alpha,r}$ in $[\alpha,\alpha+d)$ putting
$$x\sim_{\alpha,r} y \Longleftrightarrow (x|y) > \alpha - \frac{1}{2}r$$
for any $x,y\in [\alpha,\alpha+d)$. We will briefly verify that this relation is reflexive and transitive.
We have $(x|x)=d(x,x_0)\geq \alpha > \alpha - \frac{1}{2}r$ for $x\in [\alpha,\alpha+d)$.
Transitivity follows immediately from the definition of a $0$-hyperbolic space.

\begin{lem}\label{0h}
All equivalence classes of the relation $\sim_{\alpha,r}$ in $[\alpha,\alpha+d)$ are uniformly bounded by $2d+r$. Different equivalence classes are $r$-disjoint.
\end{lem}

\begin{proof}
Assume that $x\sim_{\alpha,r}y$. Then $(x|y)>\alpha-\frac{1}{2}r$, so $d(x,x_0)+d(y,x_0)-d(x,y)>2\alpha-r$. Thus
$$d(x,y)< d(x,x_0)+d(y,x_0)-2\alpha+r<2(\alpha+d)-2\alpha+r=2d+r.$$
Suppose now that $d(x,y)<r$. Then $(x|y)> \frac{1}{2}(2\alpha-r)=\alpha-\frac{1}{2}r$, which means that $x$ and $y$ belong to the same equivalence class.
\end{proof}

Observe that our result applies for all intervals $[\alpha,\alpha+d)$ and the bound $2d+r$ as well as the ``disjointness" constant $r$ do not depend on $\alpha$. Moreover, they do not depend on a particular $0$-hyperbolic space $X$.
Equivalently we can say that for any $0$-hyperbolic $X$ and $d>0$ the family $\{[\alpha,\alpha+d)\colon \alpha\in \mathbb{R}\}$ has asymptotic dimension $0$ {\em uniformly} with {\em type function} $\tau(r)=2d+r$. We refer to \cite[Example 9.14]{Roe} for the terminology.

Suppose we are given a number $R>0$. Let us consider the following families:
\begin{equation}\label{V}
\begin{split}
\mathcal{V}_0 & =  \bigcup_{n\in\mathbb{N}} [(2n-1)R,2nR)/\sim_{(2n-1)R,R},\\
\mathcal{V}_1 & =  \bigcup_{n\in\mathbb{N}_0} [2nR,(2n+1)R)/\sim_{2nR,R}.
\end{split}
\end{equation}
Both families $\mathcal{V}_0$ and $\mathcal{V}_1$ are $R$-disjoint. If two members of a fixed $\mathcal{V}_i$ are contained in different intervals, then it is easy to see that their distance is not less than $R$. For different members lying in the same interval we apply Lemma \ref{0h}.
Using Lemma \ref{0h} we also conclude that both families are uniformly $3R$-bounded. Obviously $\bigcup (\mathcal{V}_0\cup \mathcal{V}_1)=X$.

We will need some more families of subsets of $X$.
Suppose that $0<R_0<S$ and $m\in\mathbb{N}$ are given. For any $l\in\{0,\dots,2^m-1\}$ define
\begin{equation}\label{CD}
\begin{split}
\mathcal{C}_{l} & = \bigcup_{n\in\mathbb{N}_0} [(2^m (n-1)+l)S,(2^m n+l)S-R_0)/\sim_{(2^m (n-1)+l)S,R_0},\\
\mathcal{D}_{l} & = \bigcup_{n\in\mathbb{N}_0} [(2^m n +l)S-R_0,(2^m n +l)S)/\sim_{(2^m n +l)S-R_0,2^m S-R_0}.
\end{split}
\end{equation}
Similarly, $\mathcal{C}_{l}$ is $R_0$-disjoint and uniforlmy bounded for $l\in\{0,\dots,2^m -1\}$. Furthermore, $\bigcup_{l=0}^{2^m-1} \mathcal{D}_{l}$ is $(S-R_0)$-disjoint and uniformly bounded. We have also that $\bigcup(\mathcal{C}_l \cup \mathcal{D}_l)=X$ for any $l\in\{0,\dots,2^m -1\}$.

Note that the above construction works for any $0$-hyperbolic space $X$ and the boundedness and disjointness constants of the defined families remain the same.

\section{The main result}

Suppose that we have a sequence $(X_i,d_i)_{i\in\mathbb{N}}$ of $D$-discrete and $0$-hyperbolic spaces. Let $(o_i)_{i\in\mathbb{N}}$ be a sequence of basepoints $o_i\in X_i$. Define the {\em direct sum}
$$\bigoplus_{i=1}^{\infty} X_i= \left\{(x_i)\in \prod_{i=1}^{\infty} X_i \colon x_i=o_i \, \text{for all but finitely many} \, i\right\}.$$
We equip $\bigoplus_{i=1}^{\infty} X_i$ with a metric
\begin{equation}\label{metric}
d((x_i),(y_i))=\sum_{i=1}^{\infty} i\, d_i (x_i,y_i).
\end{equation}  

\begin{tw}\label{main}
The metric space $\bigoplus_{i=1}^{\infty} X_i$ has asymptotic property C.
\end{tw}

\begin{proof}
The general idea is to rewrite Yamauchi's proof (\cite{Y}) replacing original families of subsets of $\mathbb{Z}$ with families \eqref{V}, \eqref{CD} of subsets of $X_i$ for $i\in\mathbb{N}$. The families will depend on $i$, which will be denoted by a superscript $(i)$. Nevertheless, the constants of disjointness and boundedness will suit for all $i\in\mathbb{N}$ uniformly. Without loss of generality we may assume that $D=1$, i.e. all the spaces $X_i$ are $1$-discrete.

Let $R_0<R_1<\dots$ be an arbitrary increasing sequence of natural numbers. Let $k=R_0+1$, $m=R_{k2^k}+1$ and $S=R_0+R_{k2^k}$. Take a bijection $\psi\colon \{0,1,\dots,2^m -1\}\to \{0,1\}^m$.
For each $i\in\mathbb{N}$ define families $\mathcal{V}_{0}^{(i)}, \mathcal{V}_{1}^{(i)}$ of subsets of $X_i$ as in \eqref{V} with $R=R_{k2^k}$.
For each $i\in\mathbb{N}$ and $l\in\{0,1,\dots,2^m -1\}$ define similarly $\mathcal{C}_{l}^{(i)}$ and $\mathcal{D}_{l}^{(i)}$ according to \eqref{CD}.
For any $j\in\mathbb{N}$ and $l\in\{0,1,\dots,2^m -1\}$ put also
$$\mathcal{W}_{l}^{(j)}=\left\{ \prod_{i=1}^{m} V_{i} \colon V_i \in \mathcal{V}_{\psi(l)_i}^{(j+i-1)}, i\in\{1,\dots,m\}\right\},$$
where $\psi(l)_i$ stands for the $i$-th coordinate of $\psi(l)\in\{0,1\}^{m}$. 
Note that $\mathcal{W}_{l}^{(j)}$ is a family of subsets of $X_j \times \dots \times X_{j+m-1}$. Since $\mathcal{V}_{0}^{(i)}\cup \mathcal{V}_{1}^{(i)}$ covers $X_i$, we conclude that $\bigcup_{l=0}^{2^m -1} \mathcal{W}_{l}^{(j)}$ covers $X_j \times \dots \times X_{j+m-1}$ for any $j\in\mathbb{N}$.
Put
$$\mathcal{U}_0=\left\{\prod_{i=1}^{k} C_i \times W \times \prod_{i>k+m} \{y_{i}\} \colon (C_1,\dots,C_k,W)\in \bigcup_{l=0}^{2^m -1} \left(\prod_{i=1}^{k} \mathcal{C}_{l}^{(i)} \times \mathcal{W}_{l}^{(k+1)}  \right), (y_i)\in\bigoplus_{i=k+m+1}^{\infty} X_i\right\}.$$

We will show that $\mathcal{U}_{0}$ is $R_0$-disjoint. The elements of $\mathcal{U}_0$ depend on $l$, a choice of $C_1,\dots,C_k,W$, and a sequence $(y_i)$.
Suppose that $A,A'$ are different members of $\mathcal{U}_{0}$. If the corresponding sequences $(y_i), (y'_i)$ are different, then $d(A,A')\geq k+m+1>R_0$, which follows from the definition of the metric \eqref{metric} and $1$-discreteness of all $X_i$.
If $l\neq l'$, then the corresponding $W$ and $W'$ are disjoint (by construction of families $\mathcal{W}_{l}^{(k+1)}$) so $d(A,A')\geq k+1>R_0$.
In the remainig case we use the fact that each $\mathcal{C}_{l}^{(i)}$ is $R_0$-disjoint.

Let us fix a bijection $\phi \colon \{1,2,...,2^k\} \to \{0,1\}^k$. For $s\in \{1,2,\dots,k\}$ and $t\in \{1,2,\dots,2^k\}$ put
\begin{equation*}
\begin{split}
&\mathcal{U}_{2^k (s-1)+t}=\left\{\prod_{i=1}^{s-1} V_i \times D_s \times \prod_{i=s+1}^{k} V_i \times W \times \prod_{i>k+m} \{y_i\}\right. \colon \\
 & \left. V_i \in \mathcal{V}_{\phi(t)_{i}}^{(i)}, i\in\{1,\dots,s-1,s+1,\dots,k\}, (D_s,W)\in \bigcup_{l=0}^{2^m -1}\left(\mathcal{D}_{l}^{(s)}\times \mathcal{W}_{l}^{(k+1)}\right), (y_i)\in \bigoplus_{i=k+m+1}^{\infty} X_i \right\}.
\end{split}
\end{equation*}

We will show that for each $j\in\{1,2,\dots,k2^k\}$ the family $\mathcal{U}_j$ is $R_{k2^k}$-disjoint, in particular $R_j$-disjoint.
The reasoning is similar to that for $\mathcal{U}_0$. Fix $s$ and $t$.
Suppose that $A,A'$ are different members of $\mathcal{U}_{2^k(s-1)+t}$. If the corresponding sequences $(y_i), (y'_i)$ are different, then $d(A,A')\geq k+m+1>R_{k2^k}$.
If $l\neq l'$, then the corresponding $D_s$ and $D'_s$ are $S-R_0=R_{k2^k}$-disjoint so $d(A,A')\geq R_{k2^k}$.
If $l$ is common for $A,A'$ and $W\neq W'$, we use $R_{k2^k}$-disjointness of $\mathcal{W}_{l}^{(k+1)}$.
In the remainig case there exists $i\in\{1,\dots,s-1,s+1,\dots,k\}$ such that $V_i \neq V'_i$ and we use $R_{k2^k}$-disjointness of $\mathcal{V}^{(i)}_{\phi(t)_{i}}$.

Let us show that $\bigcup_{j=0}^{k2^k} \mathcal{U}_{j}$ covers $\bigoplus_{i=1}^{\infty} X_i$. Let $(x_i)\in \left(\bigoplus_{i=1}^{\infty} X_i\right) \setminus \bigcup \mathcal{U}_{0}$. Since $\bigcup_{l=0}^{2^m -1} \mathcal{W}_{l}^{(k+1)}$ covers $X_{k+1} \times \dots \times X_{k+m}$, there exist $l\in\{0,\dots,2^m -1\}$ and 
$W\in \mathcal{W}_{l}^{(k+1)}$ such that $(x_i)_{i=k+1}^{k+m} \in W$. Since $(x_i)\not\in\bigcup \mathcal{U}_{0}$, $(x_i)_{i=1}^{k}$ does not belong to any set of the form
$\prod_{i=1}^{k} C_i$ where $C_i \in \mathcal{C}_{l}^{(i)}$ for $i\in\{1,\dots,k\}$. Thus, there is $s\in\{1,\dots,k\}$ such that $x_s \not\in \bigcup \mathcal{C}_{l}^{(s)}$. Since $\bigcup(\mathcal{C}_{l}^{(s)} \cup \mathcal{D}_{l}^{(s)})=X_s$, there exists $D_s \in \mathcal{D}_{l}^{(s)}$ containing $x_s$.
Since $\bigcup (\mathcal{V}_{0}^{(i)} \cup \mathcal{V}_{1}^{(i)})=X_i$ for each $i$, there exists $t\in\{1,\dots,2^k\}$ such that
$(x_i)_{i=1}^{k} \in \prod_{i=1}^{k} V_i$ for some $V_i \in \mathcal{V}_{\phi(t)_{i}}^{(i)}$, $i\in\{1,\dots,k\}$.
Thus,
$$(x_i)\in \prod_{i=1}^{s-1} V_i \times D_s \times \prod_{i=s+1}^{k} V_i \times W \times \prod_{i>k+m} \{x_i\} \in \mathcal{U}_{2^k (s-1)+t}.$$

Finally, the families $\mathcal{U}_{j}$, $j=\{0,\dots,k2^k\}$ are uniformly bounded. It follows from the fact that the construction of these families involves families $\mathcal{V}_{0}^{(i)}$, $\mathcal{V}_{1}^{(i)}$, $\mathcal{C}_{l}^{(i)}$, $\mathcal{D}_{l}^{(i)}$, $\mathcal{W}_{l}^{(i)}$, which are uniformly bounded with meshes independent of $i\in\mathbb{N}$.
\end{proof}

\begin{rem}
In fact we have proved more, namely $\bigoplus_{i=1}^{\infty} X_i$ has {\em transfinite asymptotic dimension} not greater than $\omega$. For definition see \cite[Section 3]{Rad}. It follows from the fact that for fixed $R_0$ there exists a natural number $n=(R_0+1)2^{R_0+1}$ (depending only on $R_0$) such that for any sequence $R_0<R_1<\dots$ there exists a sequence $\{\mathcal{U}_i\}_{i=0}^{n}$ of uniformly bounded families such that $\bigcup_{i=0}^{n} \mathcal{U}_i$ covers $\bigoplus_{i=1}^{\infty} X_i$ and $\mathcal{U}_i$ is $R_i$-disjoint for $i=0,\dots,n$. Loosely speaking, the number of sufficiently disjoint and uniformly bounded families needed to cover the space is constant for fixed $R_0$.
\end{rem}

We apply our result to the case of free groups. In the realm of dicrete spaces we call a metric {\em proper} if all balls with respect to this metric are finite.

\begin{cor}
Let $(G_i)_{i\in \mathbb{N}}$ be a sequence of free groups of finite rank (which may depend on $i$) and $G=\bigoplus_{i=1}^{\infty} G_i$ their group-theoretical direct sum. Then $G$ equipped with any proper left-invariant metric has asymptotic property C.
\end{cor}

\begin{proof}
Let $d_i$ be the usual left-invariant word length metric in $G_i$ for $i\in \mathbb{N}$. Then each $(G_i,d_i)$ is $0$-hyperbolic and $1$-discrete. The metric $d$ defined by \eqref{metric} in $G$ is proper and left-invariant. By {Theorem \ref{main}} the space $(G,d)$ has asymptotic property C. The corollary follows from the fact that all proper left-invariant metrics on $G$ are coarsely equivalent and asymptotic property C is a coarse invariant.
\end{proof}

Roe proved in \cite{Roe2} that $\delta$-hyperbolic geodesic spaces with bounded growth have finite asymptotic dimension. This suggests posing the following problem:

\begin{pro}
Can Theorem \ref{main} be generalized to the case of $\delta$-hyperbolic spaces (with common $\delta$ and commonly bounded growth)?
\end{pro}

\end{document}